\documentclass[11pt]{article}

\usepackage{latexsym,amssymb,amsmath,graphics,amsthm,dsfont}
\usepackage[dvips]{graphicx}
\usepackage{authblk}
\usepackage{subcaption}
\usepackage{color}
\usepackage{float}
\newcommand\blfootnote[1]{%
	\begingroup
	\renewcommand\thefootnote{}\footnote{#1}%
	\addtocounter{footnote}{-1}%
	\endgroup
}
\makeatletter
\renewcommand\theequation%
{\thesection.\arabic{equation}}
\@addtoreset{equation}{section}
\makeatother

\def\bb{{\mathcal B}}

\def\dd{{\mathcal D}}

\def\ff{{\mathcal F}}

\def\qq{{\mathcal Q}}

\def\xx{{\mathcal X}}

\def\E{{\mathbb{E}}}

\def\K{{\mathbb{K}}}

\def\N{{\mathbb{N}}}
\def\P{{\mathbb{P}}}

\def\R{{\mathbb{R}}}

\newcommand{\wJ}{\widetilde P}

\newcommand{\norm}[1]{{\left\|{#1}\right\|}}
\newcommand{\wA}{\widehat{A}_N^{(M)}}
\newcommand{\PP}{\mathbb{P}}

\newtheorem{lemma}{Lemma}[section]
\newtheorem{proposition}[lemma]{Proposition}
\newtheorem{theorem}[lemma]{Theorem}

\newtheorem{definition}[lemma]{Definition}

\newtheorem{remark}[lemma]{Remark}

	\renewcommand{\theequation}{\thesection.\arabic{equation}}
	\title{Eigenvalue distribution of some random matrices}
	\date{}
		\author[1,2]{Mohamed Jebalia}
	\affil[1]{University of  Carthage,
	 Department of Mathematics, Research Laboratory LR21ES10, Faculty of Sciences of Bizerte,		 Bizerte, Tunisia.
}
	\author[,1]{Ahmed Souabni\thanks{Corresponding author: \texttt{souabniahmed@bizerte.r-iset.tn}}}

	\affil[2]{University of  Carthage,
		National School of Engineering of Bizerte, Bizerte, Tunisia.
	}

\begin{document}
		\maketitle
		\blfootnote{This work was supported in part by the  
			DGRST  research grant  LR21ES10.}

		\begin{abstract}
			In this paper, we investigate the eigenvalue distribution of a class of kernel random matrices whose $(i,j)$-th entry is $f(X_i,X_j)$ where $f$ is a symmetric function belonging to the Paley-Wiener space $\mathcal{B}_c$ and $(X_i)_{1\leq i \leq N}$ are i.i.d. random variables. 
			We rigorously prove that, with high probability, the eigenvalues of these random matrices are well approximated by those of an underlying estimator.   
			A particularly notable case is when $f=sinc$ , which has been widely studied due to its relevance in various scientific fields, including machine learning and telecommunications. 
			In this case, we push forward the general approach by computing the eigenvalues of the estimator. More precisely, we have proved that the eigenvalues are concentrated around zero and one. In particular, we address the case of large values of $c$ with respect to the matrix size $N$, which, to the best of our knowledge, has not been studied in the literature. Furthermore, we establish that the frequency of eigenvalues close to one is proportional to $c$. Numerical results are provided in order to illustrate the theoretical findings.
		\end{abstract}
		\maketitle
		
		\section{Introduction}
		In this work, we are interested in the study of the distribution as well as the computation of the eigenvalues of random  matrices  	$A_N \in \R^{N \times N}$ that is written as 
		$$A_N(i,j)=f(X_i,X_j)$$
		where $f$ is a real valued  function and  $X_1,\ldots,X_N$ are i.i.d random variables. The function $f$ is assumed to be symmetric implying that $A_N$
		is also symmetric. This model covers various types of random matrices, for example Euclidean matrices
		$ A_N(i,j)=g(\left\|X_i-X_j\right\|) $ and $ A_N(i,j)=g(<X_i,X_j>)$ for a given function $g$. 	The study of these matrices is motivated by the fact that they appear in many scientific areas such as big data, statistics, machine learning, geophysics, telecommunication, optics, see for example \cite{BBZ2007,BK2,DLP2013,JK2024,Skipetrov,Skipetrov2,Parisi2006,RW2005,Shawe_et_al2005}.\\
		In statistics for example, a common problem is the non-parametric regression problem having the form 
		$$Y_i=f(X_i) + \epsilon_i, \ i=1,\ldots,N$$
		where the $X_i$ are i.i.d random entries and $(Y_i)$ are the response variables.
		Among resolution methods, Kernel Ridge Regression  (KRR) computes the  estimator of the regression function $f$  as a linear combination using a positive definite kernel function $\K(\cdot,\cdot)$, that is
		$$\widehat{f}_N(.)=\sum_{i=1}^{N} \K(X_i,.) \widehat{c}_i $$
		where  the expansion coefficients vector $\widehat{c}$ is the solution of the equation 
			\begin{equation}
			\label{KRR}
			G \widehat{c}= Y, \quad  \quad Y=(Y_1,\ldots, Y_N)^T. 
		\end{equation}
	Here, the kernel (or Gram) matrix  $G$ is a random matrix positive-definite, given by
	$G(i,j)=(\K(X_i,X_j))_{1 \le i,j \le N}$.  It has been shown that the performance of KRR methods is closely related to the behaviour of the eigenvalues of the matrix $G$, see for example \cite{Amini2021,BsK2024,JKh2024}. ~\\
	Random matrices appear also in Physics. For example, wave propagation in random media in a three dimensional space  is modeled by the Helmholtz equation
	$$(\nabla^2+ c^2+i \eta)A_N(X_i,X_j)=-\frac{4\pi}{c} \delta(X_i-X_j). $$
	where  the solution is the Green function given by the random  matrix 
	$$ A_N(i,j)=\frac{\exp(i c ||X_i - X_j||)}{c||X_i-X_j||},$$
	where the $X_i$ are three dimensional random vectors and $||\cdot||$ is the Euclidean norm in $\R^3$.	The study of this wave equation is achieved through the study of the eigenvalue distribution of the real and imaginary parts of $A_N$	(see \cite{Skipetrov}).\\
	Another application of random matrices in telecommunications is given in \cite{BK2}, where the authors studied the distribution of the spectrum of a random matrix in order to estimate  the capacity of a Multi Input Multi Output (MIMO) wireless communication model.\\
	~\\
	In many applications, the order $N$ of the random matrix has to be large for improving the efficiency of the underlying methods. For example, regression problems need a large data size $N$.
	 For this reason, many studies have investigated the eigenvalues distribution of large random matrices. Mainly, three approaches have been developed to achieve this goal. \\~\\
	The first approach is to prove that the spectrum of the kernel random matrix $A_N$ remains close to the one of the integral operator associated with the same kernel \cite{AB2022,BK2,KG2000,Shawe_et_al2005}. Therefore, we rely on the extensive literature on the eigenvalues of kernel integral operators. In \cite{Shawe_et_al2005}, the authors have estimated the eigenvalues of the Gram matrix of the kernel-PCA method through the eigenspectrum of the corresponding continuous operator. 
	The studies \cite{AB2022,BK2} have investigated kernel random matrices. They have shown that their eigenvalues are close to those of  the associated integral operator, both in terms of quadratic mean and with high probability. Then, they computed the so-called degree of freedom of the associated integral operator.  At a fixed level $\epsilon>0$, the degree of freedom  of a semi-definite positive and self-adjoint Hilbert-Schmidt operator $T$ is the number of eigenvalues of $T$ that  exceed $\epsilon$.  In particular,  \cite{BK2} focused on  the sinc kernel
	$$A_N(i,j)=\frac{\sin (c(X_i-X_j)) }{\pi(X_i-X_j)} $$
	where the $X_i$ are $N$ i.i.d. random variables. In the specific case where $c \rightarrow +\infty$  and $\frac{N}{c} \rightarrow +\infty$  (implying that $c <<N$), it has been shown that the sinc kernel matrix has approximately $c$ significant eigenvalues.

~ \\
The second approach to investigate the eigenvalues distribution of the random matrix $A_N$ is to study its spectral distribution for large values of $N$. The spectral measure associated with a random Hermitian matrix $A_N \in \R^{N \times N}$ is given by
$$ \mu_{A_N}=\frac{1}{N}  \sum_{i=1}^{N} \delta_{\lambda_i(A_N)}$$
where $\lambda_{i}(A_N), i=1,\ldots,N$ are the eigenvalues of $A_N$. 
The idea is to prove that the distribution of $\mu_{A_N}$ converges to a deterministic one denoted by $\mu$ (see for example \cite{NKaroui2010}). This can be done by computing the limit of the Stieljies (or Cauchy) transform of the measure $\mu_{A_N}$. The Stieljes transform of a measure $\lambda$ is given by
$$ z \mapsto G_{\lambda}(z)= \int \frac{d \lambda(t)}{t-z}.$$
It is well known that
$$G_{\mu_{A_N}}(z)=\frac{1}{N} \mbox{Trace}\left((A_N-z I_N)^{-1}\right). $$
If the sequence of functions $(G_{\mu_{A_N}})_{N \geq 1}$ converges pointwise then $\mu_{A_N}$ converges weakly to some measure $\mu$. In the framework of kernel random matrices, N. El Karoui \cite{NKaroui2010}  has investigated this problem where the $(i,j)^{th}$ entry has the form $f\left(X_i^TX_j/p\right)$ or $f\left(\frac{||X_i-X_j ||_2^2}{p}\right)$ where $X_i \in \R^p$ are independent data vectors. The function $f$  has to be locally smooth.\\

A third way  to  study the eigenvalue distribution of a random matrix has been proposed in \cite{Skipetrov}. It has been applied in the specific case of the three dimensional sinc-kernel matrix. This method consists on the approximation of $A_N$ by the product $  H T H^T$ where $H$ is a random matrix and $T$ is a deterministic one. Then we conclude by Free probability arguments.

In this work, we investigate the eigenvalue distribution of the random matrix $A_N$using the approach introduced in \cite{Skipetrov}. A key advantage of this method lies in its generality, as it does not rely on the theory of integral operators and can be applied to any kernel $K$ within the Paley–Wiener space. In the case of the sinc kernel, we show that our findings are consistent with those previously reported in \cite{AB2022,BK2}. It is important to note that the method given in \cite{AB2022,BK2} is restricted for small values of $c$ with respect to the matrix size $N$, whereas our approach remains valid for arbitrary values of $c$.\\
More precisely, we rigorously derive and generalize the heuristic result proposed in \cite{Skipetrov} for a one-dimensional search space. Our analysis is based on approximating the matrix $A_N$by the product $\wA=HT_M H^T$ where $H \in \R^{N\times M} $ is a random matrix and $T_M \in \R^{M\times M}$ is a deterministic one. We rigorously prove the result of \cite{Skipetrov}, that is $\displaystyle A_N=\lim_{M \rightarrow +\infty} \wA$for a broad class of random matrices (see Proposition~\ref{prop:A_Decomposition}). Furthermore, using McDiarmid’s inequality and Weyl’s perturbation theorem, we establish that the eigenvalues of $A_N$
are closely approximated by those of $\wA$ with high probability (see Theorem~\ref{approx_vp}).
In the particular case of the sinc kernel, we also present several structural and approximation results concerning the matrix $T_M$
(Section~\ref{sec:sinc}). Moreover, we prove that the eigenvalues are concentrated around zero and one depending on the ratio $c/N$ (Theorem \ref{THN>M} and Theorem \ref{th:chernoff}). The frequency of eigenvalues close to one is proportional to $c$. Numerical experiments are provided following each key theoretical result in order to illustrate and validate the analysis, particularly in the sinc kernel setting.

		\section{General Approach}
		A useful approach to deal with the statistical properties of $A_N$ was developed in \cite{Skipetrov} for some kernel random matrices.  It consists on writing the matrix $A_N$ as the limit of an  estimator. We generalize  the result in \cite{Skipetrov} for any random matrix $A_N$ with entry elements $A_N(i,j)=f(X_i,X_j)$ where $f$ is a symmetric (i.e., $f(x,y)=f(y,x)$) real valued function {of $L^2([-1,1]\times[-1,1])$}. 
		\subsection{Approximation model of random matrices and convergence in average}
		To determine the eigenvalues distribution of the matrix $A_N$, we will prove in the following paragraph that it can be approximated by a simpler one. Firstly, we will define matrices that will be used throughout this paper.
		\begin{definition}
			\label{def:H_T}
		Let  $ (\phi_n)_{n \geq 0}$ be an orthonormal basis of $L^2([-1,1])$ and let
		$f \in  L^2([-1,1] \times [-1,1])$ be a symmetric real valued function and let $(X_i)_{1 \leq i \leq N}$ be $N$ i.i.d. random variables with values on  $[-1,1]$. We define the following matrices.
		\begin{itemize}
			\item 	$H$ is an $N\times M$ random matrix with elements 
			$$   H(i,m)= \phi_{m-1}(X_i), \quad \forall  i \in \{1,\ldots,N\}, \forall m \in \{1,\ldots,M\}.$$ 
			\item $T_M$ is a symmetric deterministic $M \times M$ matrix with elements 
			$$\displaystyle T_{M}(m,n) =  \int_{-1}^{1} \int_{-1}^1  
			f(x,y)
			\phi_{m-1}(x) \phi_{n-1}(y) dx dy.$$
			\item $\wA=H.T_M.H^T \in \R^{N \times N}$ is a random matrix.
		\end{itemize}
		\end{definition}
	    The approximation, in an $L^2$ sense, is given by the following proposition.
		\begin{proposition}
			\label{prop:A_Decomposition}
			Under the notations of Definition \ref{def:H_T}, we have
			\begin{equation}\label{decomposition0}
				A_N =  \wA + R_N^{(M)},
			\end{equation} 
			where 
			 {$R_N^{(M)}$ is a random matrix defined as
					$$R_N^{(M)}(i,j)= \sum_{k=0}^{M-1}< r_M(.,X_j) ;\phi_{k}(.) >_{L^2([-1,1])}  \phi_{k}(X_i) + r_M(X_i,X_j) $$	
					with 
					$r_M(x,y)=\displaystyle\sum_{k=M}^{+\infty} < f(.,y), \phi_k(.) >_{L^2([-1,1])} \phi_k(x) .$}
			Therefore
			\begin{equation}\label{decomposition}
				A = \lim_{M \to \infty} \wA.
			\end{equation} 
		\end{proposition}	
		\begin{remark}~\\
			The previous result can be easily generalized to any range $[a,b]$ instead of $[-1,1]$. Moreover, 
			 it can be generalized to the case where the $X_i$ are random vectors. This recovers  Euclidean random matrices which can be written as   $$A_N(i,j)=h(||X_i-X_j||) \quad  \mbox{or}  \quad A_N(i,j)=h(<X_i,X_j>)$$
			 for any real valued function $h$. 
		\end{remark}
	
		\begin{proof}~\\
			By applying two successive decomposition of the function $(x,y) \mapsto f(x,y)$ on the basis $(\phi_{n})$ first according to the variable $x$ then according to the variable $y$, we get:
			\begin{eqnarray}
				f(x,y) &=& \sum_{k=1}^{M} \sum_{l=1}^{M} \left[  \int_{-1}^1
				\int_{-1}^1 f(t,u) \phi_{l-1}(u)   
				\phi_{k-1}(t) du \ dt \right]  \phi_{k-1}(x) \phi_{l-1}(y) \nonumber \\ &+&  \int_{-1}^1 \sum_{k=1}^{M} \phi_{k-1}(x) \phi_{k-1}(t) r_M(t,y) dt + r_M(x,y)
			\end{eqnarray}
			with $ \displaystyle r_M(x,y)=\sum_{k=M}^{+\infty} \int_{-1}^1 f(t,y) \phi_k(t)   dt \phi_k(x) $.  
			The quantity $R_N^{(M)}$ converges to 0 since $(\phi_{n})_{n \ge 0}$ is a basis. 
		\end{proof}		
			
			From the last Proposition, we see that the approximation
	
			$$ 	A \approx \wA=H.T_M.H^T$$
			holds in an $L^2$ sense. 
		In order to validate this result, we have conducted numerical simulations in the case where $f(x,y)=\frac{\sin \left(c(x-y)\right)}{c(x-y)}$.
			Using Legendre polynomials as a basis, $c=1$ and a uniform sampling of the random variables $X_1,\ldots,X_N$ with $N=5$ on $[-1,1]$, we have computed the average value of $||A-\wA||_{HS}$ over 21 samplings of the entries $X_i$. The results are reported in Table \ref{table1:unif} and show that $\lim_{M \to \infty} ||A-\wA||_{HS}=0$.
			\begin{table}[h!]
				\centering
				\begin{tabular}{ |c|c|c|c|c|c| } 
					\hline
					M & 2 & 4 & 6 & 8 & 10\\ 
					\hline
					$||A-\wA||_{HS}$ & 0.3379 & 0.00397 & 2.35e-05 & 8.59e-08 &1.87e-10 \\ 
					\hline \hline
					M & 12 & 14 & 16 & 18 & 20 \\ 
					\hline
					$||A-\wA||_{HS}$ & 3.27e-13 & 4.70e-15 & 4.26e-15 & 5.51e-15 & 6.60e-15 \\ 
					\hline
				\end{tabular}
				\caption{$||A-\wA||_{HS}$ as a function of $M$ using the uniform law.}
				\label{table1:unif}
			\end{table}
			~\\

			The  expected error $\E[||R_N^{(M)}||]$ resulting from the approximation of $A$ by $\wA$ is estimated in the following statement.

			\begin{lemma}~\\
				\label{prop:R_M_c_const}
				Let $f \in L^2([-1,1]^2)$ and let $||.||_{HS}$ be the Hilbert-Schmidt matrix norm.Then
				$$\E \left[||R_N^{(M)}||_{HS}\right] \leq N(M+1) ||r_M||_{L^2([-1,1]^2)}.$$
			\end{lemma}
			
			\begin{proof}
				{
					From the decomposition established in Proposition~\ref{prop:A_Decomposition}, we have
					$$||R_N^{(M)}||_{HS}=\sqrt{\sum_{1 \leq i,j \leq N} (R_N^{(M)})^2(i,j)} 
					=\sqrt{\sum_{i,j=1}^{N}   
						\left[ \sum_{k=0}^{M-1} <r_M(\cdot,X_j);\phi_{k}(\cdot)>\phi_{k}(X_i)
						+r_M(X_i,X_j)   \right]^2 }.
					$$
					Using Minkowski inequality, one gets
					$$||R_N^{(M)}||_{HS} \leq 
					\sum_{k=0}^{M-1} \sqrt{\sum_{i,j=1}^{N}  <r_M(\cdot,X_j);\phi_{k}(\cdot)>^2\phi_{k}^2(X_i)} 
					+ \sqrt{\sum_{i,j=1}^{N} r_M(X_i,X_j)^2}.	 $$
					By taking the expectation of the previous inequality, one obtains
					$$\E\left[||R_N^{(M)}||_{HS}\right] \leq 
					\sum_{k=0}^{M-1} \E \sqrt{\sum_{i,j=1}^{N}  <r_M(\cdot,X_j);\phi_{k}(\cdot)>^2\phi_{k}^2(X_i)} 
					+ \E \sqrt{\sum_{i,j=1}^{N} r_M(X_i,X_j)^2}.	 $$
					The inequality $\E[\sqrt{t}] \leq  \sqrt{\E[t]}$ implies that
					$$\E\left[||R_N^{(M)}||_{HS}\right]_{H-S} \leq 
					\sum_{k=0}^{M-1} \sqrt{\sum_{i,j=1}^{N}  
						\E\left[ <r_M(\cdot,X_j);\phi_{k}(\cdot)>^2\phi_{k}^2(X_i)\right]} 
					+ \sqrt{\sum_{i,j=1}^{N}\E\left[ r_M(X_i,X_j)^2\right]}.	 $$
					We notice now that
					\begin{equation}
						\begin{aligned}
							\E_{X_i,X_j}\left[ <r_M(\cdot,X_j);\phi_{k}(\cdot)>^2\phi_{k}^2(X_i)  \right] & =
							\E_{X_j}\left[ <r_M(\cdot,X_j);\phi_{k}(\cdot)>^2  \right]
							\underbrace{\E_{X_i}\left[ \phi_{k}^2(X_i)  \right]}_{=1}\\
							& \leq \E_{X_j}\left[ ||r_M(\cdot,X_j)||_{L^2([-1,1])}^2 \underbrace{||\phi_k||_{L^2([-1,1])}^2}_{=1}\right]\\
							&=||r_M||_{L^2([-1,1]^2)}^2.
						\end{aligned}
					\end{equation}
					Moreover $\E_{X_i,X_j}\left[ r_M^2(X_i,X_j) \right]=||r_M||_{L^2([-1,1]^2)}^2.$\\
					Consequently $\E\left[||R_N^{(M)}||_{HS}\right] \leq (M+1)N||r_M||_{L^2([-1,1]^2)}$. 
				}
			\end{proof}

			In the following section, we will show that the eigenvalues of $A_N$ are well approximated by those of the  matrix $\wA$.
			\subsection{Case of random matrices related to band-limited kernels and convergence of eigenvalues in high probability}
			In the sequel, let $c$ be a positive real number. Let $\ff(f)$
			be the Fourier transform of $f$  given by
			$$\ff(f)(x)=\frac{1}{\sqrt{2\pi}} \int_{\R} e^{-itx}f(t) dt. $$
			We define the Paley-Wiener space $ \bb_c = \{ f\in L^2(\R) : \ff f \in [-c,c] \} $ and the space $\dd$ of $[-1,1]$-time limited functions. Let $B_c$ and $D$ be the projection from $L^2(\R)$ onto $\bb_c$ and $\dd$ respectively :
			$$ B_c f := \ff^{-1} {\mathds 1}_{[-c,c]} \ff f \qquad D f = f(x) {\mathds 1}_{[-1,1]} (x).$$ 
			
			Throughout this section, we assume that $f \in B_c$.
			
			Next, we recall that  Legendre polynomials can be defined, for example, through the Rodriguez formula, 
			$$ P_n(x)=\frac{1}{2^n n!} \frac{d^n}{dx^n}(1-x^2)^n, \ x \in I=[-1,1]. $$
			The normalized Legendre polynomials is given by
			$$ \displaystyle \widetilde{P}_n (x)= \sqrt{n+\frac{1}{2}} P_n(x).$$
			The family $\{\widetilde{P}_n\}$ forms an orthonormal basis of $L^2(I)$. In the sequel, we expand $f$ in the basis of the normalized Legendre polynomials. We denote  $\pi_n$ its associated projection  kernel given by  $$\pi_n(x,y)=\sum_{k=0}^{n-1}\phi_k(x) \phi_k(y).$$ 
			We also  denote the associated integral operator by  $\Pi_n$ and  $r_M(f)=f-\Pi_M(f)$.

			Lemma~\ref{prop:R_M_c_const} induces that if $\norm{r_M(f)}$ has a sufficient decay rate then $\E \left[\norm{R_N^{(M)}} \right]$ converges to zero. In the following, we study the decay of $\norm{r_M(f)}$. \\

			\begin{proposition}
				\label{prop1}
				Suppose that, for all $y \in [-1,1]$, the function $x \mapsto f(x,y) \in B_c$.
				Under the condition $M > \frac{ec}{2}$  , we have
				\begin{equation}
					\norm{r_M(f)}_{L^2([-1,1]^2)}=\norm{f-{\Pi_M}(f)}_{L^2([-1,1]^2)} \leq 
					\frac{2\sqrt{\pi}e^{-3/2}}{\sqrt{3}c} \left(\frac{ec}{2M}\right)^{M+\frac{3}{2}}
					\frac{1}{\sqrt{1-\left(\frac{ec}{2M}\right)^2}}
					||f||_{L^2([-1,1]^2)}.
				\end{equation}
		
			\end{proposition}
	
			\begin{proof} 
				We have 
				$$ (f -\Pi_M (f))(x,y) =\sum_{k \geq M} <f(.,y), \wJ_{k}(.)> {\wJ_k}(x).$$
			
				{
					By integrating with respect to '$x$', one gets
					\begin{equation}
						\label{eq:r_M_norm_yfixed0}
							\int_{x \in [-1,1]} |(f -\Pi_M (f))(x,y)|^2  dx=
							||(f -\Pi_M (f))(\cdot,y)||_{L^2(-1,1)}^2  
							= S_1+S_2
						\end{equation}
					with
					\begin{equation}
							S_1=\sum_{k\neq j \geq M} <f(.,y), \wJ_{k}(.)>  <f(.,y), \wJ_{j}(.)>
							\int_{-1}^1 \wJ_k(x) \wJ_j(x) dx =0
						\end{equation}
					and
					\begin{equation}
						S_2= 
							\sum_{k \geq M} <f(.,y), \wJ_{k}(.)>^2	\int_{-1}^1 \wJ_k(x)^2 dx=\sum_{k \geq M} <f(.,y), \wJ_{k}(.)>^2.
					\end{equation}
				Then we get
				\begin{equation}
						\label{eq:r_M_norm_yfixed}
					\int_{x \in [-1,1]} |(f -\Pi_M (f))(x,y)|^2  dx=\sum_{k \geq M} <f(.,y), \wJ_{k}(.)>^2.
				\end{equation}
					Using Lemma 3.2 in \cite{JATpaper} in the particular case of Legendre polynomials, we can write, for {$f \in B_c$ and for} fixed $y$ 
					$$|<f(\cdot,y), P_{k}(\cdot)>_{L^2(-1,1)} | \leq 
					\frac{\sqrt{2\pi}e^{-3/2}}{c}
					\left(\frac{ec}{2k}\right)^{k+3/2} ||f(\cdot,y)||_{L^2(-1,1)}. $$
					Then 
					\begin{equation}
						\label{eq:scalar_prod_upbound}
						\begin{aligned}
							|<f(\cdot,y), \wJ_{k}(\cdot)>_{L^2(-1,1)} | & \leq 
							\frac{\sqrt{2\pi}e^{-3/2}}{c} \frac{1}{\sqrt{k+1/2}}
							\left(\frac{ec}{2k}\right)^{k+3/2} ||f(\cdot,y)||_{L^2(-1,1)}\\
							& \leq 
							\frac{b}{c} 
							\left(\frac{ec}{2k}\right)^{k+3/2} ||f(\cdot,y)||_{L^2(-1,1)}.
						\end{aligned}
					\end{equation}
					with $b=\frac{2\sqrt{\pi}e^{-3/2}}{\sqrt{3}}. $
					By combining these results and integrating with respect to '$y$', one obtains
					$$ ||f- \Pi_M (f)||_{L^2([-1,1]^2)}^2 \leq
					\left(  \frac{b}{c} ||f||_{L^2([-1,1]^2)} \right)^2
					\sum_{k \geq M} 
					\left(\frac{ec}{2k}\right)^{2k+3}. 
					$$
				}
				It follows from the condition $M>ec/2$ that
				\begin{equation}
					\begin{aligned}
						||f -\Pi_M (f)||_{L^2([-1,1]^2)}^2 & \leq  
						\left(  \frac{b}{c} ||f||_{L^2([-1,1]^2)} \right)^2 
						\sum_{k \geq M} 
						\left(\frac{ec}{2M}\right)^{2k+3}\\
						& =
						\left(  \frac{b}{c} ||f||_{L^2([-1,1]^2)}  \left(\frac{ec}{2M}\right)^{M+3/2} \right)^2 
						\frac{1}{ 1-\left(\frac{ec}{2M}\right)^{2}}.
					\end{aligned}
				\end{equation}
			\end{proof}
			{
				The following result is an immediate consequence of
				Proposition~\ref{prop:R_M_c_const} and Proposition~\ref{prop1}.
				\begin{proposition}~\\
					\label{th:R_M_c_const}
				We denote by $||.||_{HS}$ the Hilbert-Schmidt matrix norm. We have for  $M > \frac{ec}{2}$
					$$\E \left[\norm{R_N^{(M)}}_{HS}\right] \leq N(M+1)
					\frac{2\sqrt{\pi}e^{-3/2}}{\sqrt{3}c}
					\left(\frac{ec}{2M}\right)^{M+\frac{3}{2}}  \frac{1}{\sqrt{1-\left(\frac{ec}{2M}\right)^2}}  ||f||_{L^2([-1,1]^2)}.$$
					Therefore
					$$\lim_{M \to \infty} \E \left[\norm{R_N^{(M)}}_{HS}\right]=0.  $$
				\end{proposition}
				
			}
			
			The following result states that, with high probability, the eigenvalue distribution of the matrix $A_N$ are well approximated by those of $\wA$.  It is a consequence of Proposition~\ref{th:R_M_c_const} and the following McDiarmid's inequality (see Theorem 6.5 in \cite{Book:BLM2013}). \\
			{\bf  McDiarmid's concentration inequality}~\\
			Let $f : \xx_1 \times \ldots \times \xx_n \to \R$ satisfy the bounded differences property with bounds $c_1, \ldots, c_n$ {i.e., $\forall \ i \in \{1,\ldots,n\}$
				$$\sup_{x_1,\ldots,x-n,x_i'} \ 
				|f(x_1,\ldots, x_n) - f(x_1,\ldots,x_i',\ldots,x_n)| \leq c_i. $$}
			Consider independent random variables $X_1 \in \xx_1, \cdots, X_n \in \xx_n$. Then for any $\varepsilon>0$,
			$$ \P\left(| f(x_1,\ldots, x_n) - \E(f(x_1,\ldots, x_n)) > \varepsilon |\right) \leq 2 \exp \left(-\frac{2\varepsilon^2}{c_1^2 + \ldots + c_n^2}\right).$$
			\begin{theorem}\label{approx_vp}
				 Let $c>0$, $\varepsilon >0$ and ${f \in B_c}$.
				For $M > \frac{ec}{2}$, there exists $N_\varepsilon$ such that for all $N>N_\varepsilon$, the eigenvalues of the matrix $A_N$ satisfy
				\begin{equation}
					|\lambda_j(A_N)|  \leq \epsilon
					+ 	 N(M+1)		\frac{2\sqrt{\pi}e^{-3/2}}{\sqrt{3}c} \left(\frac{ec}{2M}\right)^{M+\frac{3}{2}}
					\frac{1}{\sqrt{1-\left(\frac{ec}{2M}\right)^2}}
					||f||_{L^2([-1,1]^2)}
					+ |\lambda_{j}(\wA)|.
				\end{equation}
				with a probability at least $1-2\exp\left(-\frac{2\epsilon^2}{D}\right)$ with $D:=\frac{64 \pi  B^2}{3 e^3 c^2} \left(\frac{ec}{2M}\right)^{2M+3}
				\left[\frac{(M+1)(2N-1)}{1 - \frac{ec}{2M}}\right]^2$ and $B=\norm{f}_\infty$.
			\end{theorem}
			
			\begin{proof}
				By using the symmetry of $A$ and $\wA$, the 
				fact that $\rho(K) \leq ||K||$ for every matrix $K$ and every matrix norm
				combined with 
				Weyl's inequality (see  Theorem III.2.1 in \cite{Book:Bhatia97})   with $i=1$ and $1  \leq j \leq N$ gives
				\footnote{{We suppose that $\lambda_1(A) \geq \ldots  \geq\lambda_N(A)$.}} :
				\begin{equation}
					\label{eq:maj_Weyl}
					\begin{aligned}
						|\lambda_j(A)| \leq |\lambda_1(A-\wA)| + |\lambda_{j}(\wA)| & \leq ||A-\wA||_2+ |\lambda_{j}(\wA)| \\
						& =\rho(A-\wA)+ |\lambda_{j}(\wA)|\\
						& \leq ||A-\wA||_{HS} + |\lambda_{j}(\wA)|.
					\end{aligned}
				\end{equation}
				Then, we use McDiarmid's inequality to give an approximation of  $ \norm{R_N^{(M)}}_{HS}$ by its expectation with high probability. To do so, it suffices to prove that $\norm{R^{(M)}_N}$ satisfy the bounded difference property
				$$ \lvert ~||R_N^{(M)}||_{HS} - ||\widetilde{R}_N^{(M)}||_{HS}  \rvert \leq c_i$$
				where  $||R_N^{(M)}||_{HS}=||R_N^{(M)}(X_1,\ldots,X_i,\ldots,X_N)||_{HS}$, 
				$||\widetilde{R}_N^{(M)}||_{HS}=||R_N^{(M)}(X_1,\ldots,X'_i,\ldots,X_N)||_{HS}$
				and   $c_i$ is a constant to be determined. Note that it is sufficient to show that
				$$ \sum_{1 \leq k,l \leq N}   |R_N^{(M)}(k,l) - \widetilde{R}_N^{(M)}(k,l)| \leq c_i.$$
				For the case where $k \neq i$ and $l \neq i$, we have $R_N^{(M)}(k,l) - \widetilde{R}_N^{(M)}(k,l)=0$.\\
				For the case where $k=i$ with $l \neq i$, we have
				$$	  	 R_N^{(M)}(i,l) - \widetilde{R}_N^{(M)}(i,l)= \sum_{k=0}^{M-1} g_k(X_l) 
				\left[\wJ_k(X_i) - \wJ_k(X'_i) \right] + r_M(X_i,X_l)  - r_M(X'_i,X_l),
				$$
				with $$g_k(X_l):= <r_M(\cdot,X_l) ; \wJ_{k}(\cdot)> \leq ||r_M(\cdot,X_l)||_{L^2([-1,1])}. $$
				Therefore
				\begin{equation}
					\label{eq:R-R'_ineqj}
					| R_N^{(M)}(i,l) - \widetilde{R}_N^{(M)}(i,l)| \leq 2 M ||r_M(\cdot,X_l) ||_{L^2([-1,1])}
					+|r_M(X_i,X_l)|+   |r_M(X'_i,X_l)|. 
				\end{equation}
				Let us find an upper bound for the function $(x,y) \mapsto r_M(x,y)$. Using (\ref{eq:scalar_prod_upbound}) and the inequality $M>\frac{ec}{2}$, we get
				\begin{equation}
					\begin{aligned}
						|r_M(x,y)| & =|\sum_{k \geq M} < f(\cdot,y), \wJ_{k}(\cdot)> \wJ_{k}(x)|
						\leq  \sum_{k \geq M} |< f(\cdot,y), \wJ_{k}(\cdot)>|\\
						& \leq \frac{b}{c}  ||f(\cdot,y)||_{L^2(-1,1)} \sum_{k \geq M}
						\left(\frac{ec}{2k}\right)^{k+3/2} \\
						& \leq \frac{b}{c}  ||f(\cdot,y)||_{L^2(-1,1)}
						\left(\frac{ec}{2M}\right)^{M+3/2}\frac{1}{1 -\frac{ec}{2M}}
					\end{aligned}
				\end{equation} 
				with $b=\frac{2 \sqrt{\pi}e^{-3/2}}{\sqrt{3}}$. Since the function $f$ is upper bounded by $B$, we obtain
				\begin{equation}
					\label{eq:r_M_upper_bound}
					|r_M(x,y)| \leq \frac{2b B}{c} 	\left(\frac{ec}{2M}\right)^{M+3/2}\frac{1}{1 -\frac{ec}{2M}}.
				\end{equation}
				Now, we give interest to the quantity $||r_M(\cdot,X_l)||_{L^2([-1,1])}$. Using (\ref{eq:r_M_norm_yfixed}) and (\ref{eq:scalar_prod_upbound}), we have
				\begin{equation}
					||r_M(\cdot,X_l)||_{L^2([-1,1])}^2= 	\sum_{k \geq M} <f(.,y), \wJ_{k}(.)>^2
					\leq \sum_{k \geq M} \frac{b^2}{c^2} \left(\frac{ec}{2k}\right)^{2k+3} ||f(\cdot,y)||_{L^2(-1,1)}^2
				\end{equation} 
				Using again the boundedness of $f$, we obtain for $M>\frac{ec}{2}$
				\begin{equation}
					\label{eq:norm_rM_X_j_up-bound}
					||r_M(\cdot,X_l)||_{L^2([-1,1])}
					\leq  \frac{2b B}{c} \left(\frac{ec}{2M}\right)^{M+3/2} 
					\frac{1}{\sqrt{1-\left(\frac{ec}{2M}\right)^2}}
					\leq 
					\frac{2b B}{c} \left(\frac{ec}{2M}\right)^{M+3/2} 
					\frac{1}{1-\frac{ec}{2M}}.
				\end{equation} 
				By combining (\ref{eq:R-R'_ineqj}), (\ref{eq:r_M_upper_bound}) and (\ref{eq:norm_rM_X_j_up-bound}), we get when $l \neq i$
				\begin{equation}
					\label{eq:RM_i_l}
					| R_N^{(M)}(i,l) - \widetilde{R}_N^{(M)}(i,l)| \leq \frac{4  b B}{c}  \left(\frac{ec}{2M}\right)^{M+3/2} 
					\frac{M+1}{1-\frac{ec}{2M}}.
				\end{equation}
				For the case where $i=l$, we have
				$$
				| R_N^{(M)}(i,i) - \widetilde{R}_N^{(M)}(i,i)| \leq M ||r_M(\cdot,X_i)||_{L^2([-1,1])} + |r_M(X_i,X_i)|
				+  M ||r_M(\cdot,X'_i)||_{L^2([-1,1])} + |r_M(X'_i,X'_i)|.
				$$
				Using again (\ref{eq:r_M_upper_bound}) and (\ref{eq:norm_rM_X_j_up-bound}), we obtain
				\begin{equation}
					\label{eq:RM_i_i}
					| R_N^{(M)}(i,i) - \widetilde{R}_N^{(M)}(i,i)| \leq \frac{4 b B}{c}  \left(\frac{ec}{2M}\right)^{M+3/2} 
					\frac{M+1}{1-\frac{ec}{2M}}.
				\end{equation}
				Finally, by the symmetry of the matrices $R_N^{(M)}$ and $\widetilde{R}_N^{(M)}$, we have for $l \neq i $
				$$R_N^{(M)}(l,i) - \widetilde{R}_N^{(M)}(l,i)=R_N^{(M)}(i,l) - \widetilde{R}_N^{(M)}(i,l).$$
				The last equation combined with (\ref{eq:RM_i_l}) and (\ref{eq:RM_i_i}) gives
				$$ \sum_{1 \leq k,l \leq N}   |R_N^{(M)}(k,l) - \widetilde{R}_N^{(M)}(k,l)| \leq 
				\frac{4  b B}{c}  \left(\frac{ec}{2M}\right)^{M+3/2} 
				\frac{(M+1)(2N-1)}{1-\frac{ec}{2M}}
				.$$
			\end{proof}
		Following Theorem~\ref{approx_vp}, the next step of the study requires the estimation of the eigenvalues of the matrix $\wA$. In order to do so, we restrict ourselves to the sinc kernel case. 	
	\section{The sinc kernel case}
	\label{sec:sinc}
		In this section, we suppose that $A_N$ is the real symmetric $ N \times N $ matrix with entries $$ A_N(i,j)=\frac{\sin(c(X_i-X_j))}{c(X_i-X_j)}. $$
	Here $X_i$ are i.i.d. random points following the uniform law in $I=[-1,1]$.
Since the sinc function is  real, continuous, positive definite and even, then $A_N$ is a symmetric positive and semi-definite matrix.
Gershgorin’s theorem immediately implies that 
$Sp(\frac{A_N}{N}) \subset [0,1]$."
	\\
	In the following section, we provide some properties of the matrix $T_M$ that are valid for any basis $(\phi_{n})$.
	\subsection{General properties of the matrix $T_M$}
	\label{sec:prolate}
	 In order to derive the general properties of the matrix $T_M$, we will first recall some facts about the time-frequency limiting operator.\\ 
	Let $c>0$. Since, $\displaystyle K_c(x,y) = \frac{\sin(c(x-y))}{\pi(x-y)}$ is a Mercer's kernel, then it can be written as
	\begin{equation}\label{Mercer}
		K_c(x,y) = \sum_{n=0}^\infty \lambda_{n}(c) \psi_{n,c}(x) \psi_{n,c}(y).
	\end{equation} 
	Here $(\lambda_{n}(c))_{n \geq 0}$ are the eigenvalues of the Sinc integral operator related to the kernel $K_c$ that we denote $\qq_c$; and $(\psi_{n,c})_{n \geq 0}$ are the associated eigenfunctions  known in the literature as Slepian's functions or prolate spheroidal wave functions (PSWFs) For more details on these functions we refer reader to \cite{Sle1,Wang4}.
	Let us consider the Fredholm integral equation : 
	$$ \mbox{Find $\lambda \in \R$ and $\psi\in L^2([-1,1])$ :} \ \qq_c.\psi(x)=\int_{-1}^1 \frac{\sin c(x-t)}{\pi(x-t)} \psi(t) dt = \lambda \psi(x), \quad x\in(-1,1).$$ 
	The operator $\qq_c$ is of type Hilbert-Schmidt. Then, it has a countable set of eigenvalues and corresponding eigenfunctions satisfying the following properties:
	\begin{enumerate}
		\item The eigenvalues $\left(\lambda_{n}(c)\right)_{n\geq 0}$ are simple, real and positive. They can be ordered as 
		$$ 1>\lambda_{0}(c) > \lambda_{1}(c) > \cdots > \lambda_{n}(c) > \cdots >0. $$
		\item $\lim_{n \to \infty}\lambda_{n}(c) = 0.$
		\item The corresponding eigenfunctions $\left(\psi_{n,c}\right)_{n\geq 0}$ are orthogonal and form a complete system in $L^2(-1,1)$.
	\end{enumerate}

	We also note that $\qq_c$ writes as $\frac{2c}{\pi}\mathcal{F}^*_c \mathcal{F}_c$
where 
$\ff_c: L^2([-1,1]) \to L^2([-1,1])$ is given by
$$ \ff_c(f)(x)=\int_{-1}^1 f(t) e^{-icxt} dt= {\sqrt{2\pi}}  \ff({\mathds 1}_{[-1,1]}f)(cx), \ f \in L^2([-1,1]).$$
 We can then enounce the following properties on the coefficients of the infinite deterministic matrix $T_{M}$ that hold regardless of the basis in which $T_{M}$ is expressed. 
\begin{proposition}
	For any orthonormal basis $(\phi_{n})_{n \geq 0}$, the matrix $T_{M}$ has the following properties:
	\begin{enumerate}
		\item $\forall \ k,j \in \{1,\ldots,M\}$, we have 
			\begin{equation}\label{expr_T}
				\displaystyle T_{M}(k,j) =
				\frac{ \pi }{c}<\qq_c\phi_{k-1},\phi_{j-1}>= \frac{ \pi }{c}   <\mathcal{F}_c\phi_{k-1}, \mathcal{F}_c\phi_{j-1}>.
		\end{equation}
		\item $\forall\  k,j \in \{1,\ldots,M\}$, we have 
		\begin{equation}
			\label{expr2_T}
			T_{M}(k,j)= \frac{\pi}{c} <DB_cD \phi_{k-1},\phi_{j-1} >_{L^2(\R)}.
		\end{equation}
		\item $T_{M}$ is positive.
		\item For all $k,j \in \{1,\ldots,M\}$
		\begin{equation}
			\label{eq:Tkj}
			T_{M}(k,j) = \frac{\pi}{c}\sum_{n=0}^\infty \lambda_n(c) \beta^n_{k-1}(c)\beta^n_{j-1}(c)
		\end{equation}
		where $\lambda_n(c)$ are  the eignevalues of $\qq_c$ and $\displaystyle \beta^n_k(c) = \int_{-1}^1 \psi_{n,c}(x) \phi_k(x) dx. $
	\end{enumerate}
\end{proposition}

\begin{proof}~\\
	\begin{enumerate}
		\item  We have:
		\begin{equation}
			\begin{aligned}
				\displaystyle T_{M}(k,j) & =  \int_{-1}^{1} \int_{-1}^1 \frac{\sin(c(x-y))}{c(x-y)} \phi_{k-1}(x) \phi_{j-1}(y) dx dy\\
				&=
				\frac{\pi }{c}<\qq_c\phi_{k-1},\phi_{j-1}>= \frac{\pi }{c}   <\mathcal{F}_c\phi_{k-1}, \mathcal{F}_c\phi_{j-1}>.
			\end{aligned}
		\end{equation}
		\item The proof of the second point is immediate.
		\item Let $x \in \R^M$, we have
		\begin{equation}
			\begin{aligned}
				<T_M x,x>=\sum_{1 \leq i,j \leq M} T_M(i,j) x_i x_j & = 
				\sum_{1 \leq i,j \leq M} \frac{ \pi}{c} <\qq_c\phi_{i-1}, \phi_{j-1}> x_i x_j\\
				&	= \frac{\pi}{c}\sum_{1 \leq i,j \leq M}  <\qq_c \left(x_i \phi_{i-1}\right), \left(x_j\phi_{j-1}\right)>. 
			\end{aligned}
		\end{equation}
		Let $\displaystyle \widetilde{\phi}:= \sum_{i=1}^M x_i \phi_{i-1} \in L^2([-1,1])$. We obtain
		$$	<T_Mx,x>= \frac{\pi}{c}  < \qq_c \left(\sum_{i=1}^{M}x_i \phi_{i-1}\right), \sum_{j=1}^M \left(x_j\phi_{j-1}\right)>
		=\frac{ \pi}{c}<\qq_c \widetilde{\phi},\widetilde{\phi}> \geq 0.$$
		The positivity of the last term is due to the positivity of the operator $\qq_c$.
		\item This point results from the combination of  \eqref{Mercer} and \eqref{expr_T}.
	\end{enumerate}
\end{proof}

{By \eqref{expr_T}, we see that the infinite matrix $ \displaystyle T_{\infty}=\lim_{M\rightarrow +\infty} T_M$ can be seen as the matrix of the operator $\qq_c$ in the orthonormal basis $(\phi_n)_n$. Obviously, the ideal choice of such a basis is the one that diagonalizes $\qq_c$ which is the  PSWFs basis. In this special case, the matrix $T_M$ is diagonal and its diagonal contains the values $\frac{\pi}{c}\lambda_{0}(c),\ldots, \frac{\pi}{c}\lambda_{M-1}(c)$. In the following section, we study the eigenvalues distribution of the corresponding estimator $\wA$.

	\subsection{Eigenvalues distribution of $A_N$ in the case where $c$ is fixed}
	As the PSWFs basis diagonalizes the operator under consideration, we choose to represent our estimator in this basis. 
	It is interesting to note that Theorem \ref{approx_vp} is still valid when the estimator $\wA$ is expressed in the basis of prolate spheroidal wave functions (see for example Proposition 3 in \cite{BK17}).\\
	In this case, the matrix $T_M$ is diagonal and positive.  We have then the following result on the eigenvalues of the estimator $\wA$. 
	
	\begin{theorem}\label{THN>M}
		Let $A_N$ be the sinc kernel random matrix.	
		Suppose that the sample size is { larger than the trunction index  ($N  \geq M$)}, then the eigenvalues of $\wA$ expressed in the PSWFs basis are such that
		\begin{itemize}
			\item $0$ of multiplicity $N-M$. 
			\item The eigenvalues different from zero of
			{
			 $\frac{1}{N}\wA$ converge to the eigenvalues of $T_M$, almost surely,  as $N$ grows to infinity}.
		\end{itemize}
		
	\end{theorem}
	\begin{proof}
		In this case $T_M$ is diagonal and definite positive (since its diagonal contains the positive values $\frac{\pi}{c}\lambda_0(c),\ldots,\frac{\pi}{c}\lambda_{M-1}(c)$). Therefore, the matrix  $ \wA = HT_M H^T$ is also definite positive. We first notice that 
		$$ \wA =  \left(T_M^{1/2}H_M^T\right)^T\left(T_M^{1/2}H^T\right) = B^TB$$
		with $B= T_M^{1/2}H^T \in \R^{N \times M}$.
		Therefore, the eigenvalues of $B^TB$ different from zeros are the same as the eigenvalues of $BB^T \in \R^{M\times N}$. Consequently, the eigenvalues of $\wA$ are:
		\begin{itemize}
			\item $0$ of multiplicity $N-M$.
			\item the eigenvalues of $BB^T = T_M^{1/2}H^T H T_M^{1/2}$.
		\end{itemize}
		Let us compute $B B^T$. 	We use the following notations:
		
		$$ H^T = \left[H_1^T,H_2^T, \cdots , H_N^T\right], \qquad \mbox{with} \quad H_i^T = \left[\varphi_0(X_i), \cdots , \varphi_{M-1}(X_i)\right]^T, \quad \forall \ i=1,\ldots,N.$$
		It follows that
		$$ B=T_M^{1/2}H^T = \left[T_M^{1/2}H_1^T, \cdots, T_M^{1/2}H_N^T\right] = \left[B_1, \cdots, B_N\right]. $$
		Since the matrix expressed in the PSWFs basis $T_M$ is diagonal, one gets 
		$$\left(B_i B_i^T\right)(k,\ell) = \frac{\pi}{c}\sqrt{\lambda_{k-1} \lambda_{\ell-1}} \psi_{k-1,c}(X_i) \psi_{\ell-1,c}(X_i),$$
		where the $(\lambda_k(c))$ are the eigenvalues of the integral operator $\qq_c$  and $(\psi_{k,c})$ are the associated PSWFs functions. By the orthogonality of the $\psi_{k,c}$, one gets 
		$$ \mathbb{E}\left[B_i B_i^T\right] = T_M.$$
		Noticing that $\displaystyle \frac{1}{N}BB^T=\frac{1}{N}\sum_{i=1}^N B_i  B_i^T$, we deduce  
 by the strong law of large numbers, as $N\to \infty$ and for fixed $M$, that
		$$ \frac{1}{N }B B^T =\frac{1}{N}\sum_{i=1}^N B_i B_i^T \longrightarrow \mathbb{E}[B_1B_1^T]= T_M. $$	
	\end{proof}
	
	To examine the non-zero eigenvalues of $\wA$, we follow the same approach as in Theorem \ref{THN>M}.  As illustrated in Figure \ref{fig:nonzeroev}, we observe that the positions of the eigenvalues of $\wA$ are closely aligned with  those of the matrix $T_M$. This validates Theorem \ref{THN>M}.
	\begin{figure}[h]
		\centering
		\begin{subfigure}{0.41\textwidth}
			\centering
			\includegraphics[width=\linewidth]{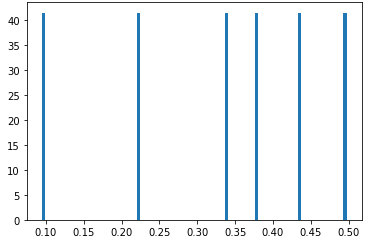} 
			\caption{Non-zero eigenvalues of the matrix $T_M$}
		\end{subfigure}
		\hfill
		\begin{subfigure}{0.58\textwidth}
			\centering
			\includegraphics[width=\linewidth]{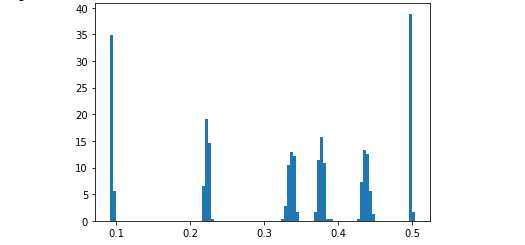} 
			\caption{Non-zero eigenvalues of the matrix $\wA$
			}
		\end{subfigure}
		\caption{Non-zero eigenvalues of the matrix $T$ (left) and non-zero eigenvalues of the matrix $\wA$ (right). These histograms have been obtained with  $c=6$ and $M=6$. For the histogram of the eigenvalues of the matrix $\wA$, the matrix is of  size  $N=10000$ and have been sampled 100 times.}
		\label{fig:nonzeroev}
	\end{figure}
	
\begin{remark} ~\\
	\begin{enumerate}
		\item Theorem 3.2 is still valid for any definite positive kernel and any orthonormal basis that diagonalizes its associated integral operator.
		\item By Theorem~\ref{THN>M}, we recover the result of \cite{AB2022} and \cite{BK2}  which establish that the eigenvalues of the sinc kernel random matrix can be approximated by those of the associated integral operator.
		\item 	By the Landau and Widom's $2T\Omega$ Theorem (see \cite{Hogan}), the function $f$ is well approximated by its expansion in the PSWFs basis using its first $\frac{2c}{\pi} + \alpha \log(c)$ coefficients for some positive real number $\alpha$. Then we can choose $M=\left[\frac{2c}{\pi} + \alpha \log(c)\right]$.
		
	\end{enumerate}

\end{remark}

\subsection{Eigenvalues distribution of the estimator $\wA$ in the case where $c/N$ is bounded}
In this case, the appropriate choice is the basis of Hermite functions.
It may be useful to recall some well known facts about these functions. 
The Hermite functions are defined by
$$\varphi_n(x) = \alpha_n H_n(x)e^{-x^2}; \qquad \alpha_n = \frac{1}{\pi^{1/4} 2^{n/2} \sqrt{n!}} $$
where $H_n$ are the Hermite polynomials given by 
$ H_n(x)=(-1)^n e^{x^2} \frac{d^n}{dx^n}\left[e^{-x^2}\right]$. 
The family $(\varphi_n)_n$ is an orthonormal basis of $L^2(\R)$. We use here a scaled version of Hermite functions given by $$ \varphi^{(c)}_n(t)= c^{1/4} \varphi_n (\sqrt{c}t). $$
Obviously, $(\varphi^{(c)}_n)$ is still an orthonormal basis of $L^2(\R)$.
In \cite{JATpaper}, authors have proven that, for any $f\in \mathcal{B}_c$,
\begin{equation}\label{hermite_err}
	\norm{f-K^{(c)}_n(f)}_{L^2(-\sqrt{c},\sqrt{c})} \leq \frac{34 c^{3/2}}{\sqrt{2n+1}} \norm{f}_{L^2(\R)},
\end{equation}
where $ K^{(c)}_n(f) := \displaystyle \sum_{k=0}^{n}<f,\varphi^{(c)}_k>_{L^2(\R)}\varphi^{(c)}_k(x)$.
To support equation \eqref{hermite_err}, we conducted numerical simulations using a Python program giving Figure \ref{fig:Hermite} for $c= 100$ and $n=30$. 
\begin{center}
	\begin{figure}[H]
		\includegraphics[scale=0.4]{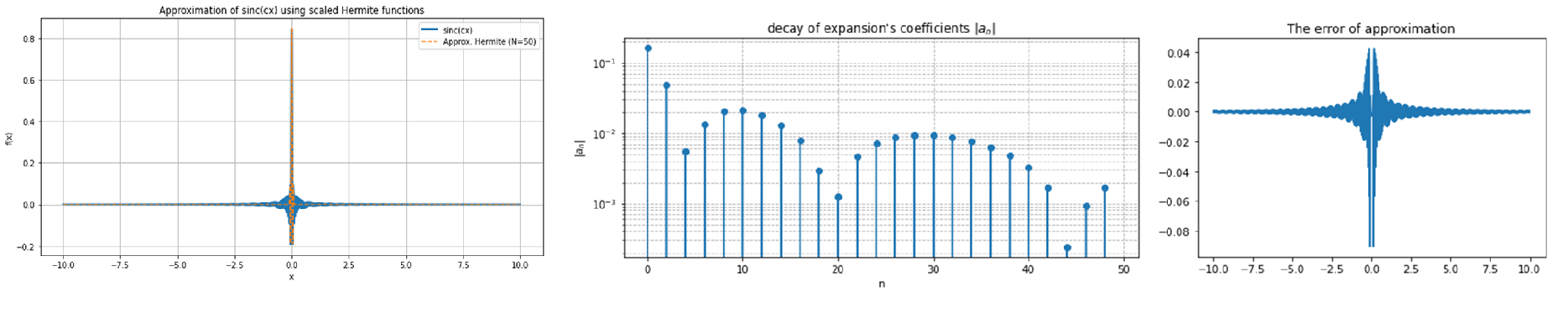}
		\caption{The approximation of $f=sinc$ using 30 scaled Hermite functions for c=100.}
		\label{fig:Hermite}
	\end{figure}
\end{center}
According to these numerical results, one may conjecture that the actual reconstruction of $f$ by $K_n^{(c)}f$ is much more precise than the one given by \eqref{hermite_err}. Based on these numerical approximations, we assert that the sinc function can be accurately approximated for large values of $c$ using a finite and reasonably small number of re-scaled Hermite functions. \\
In light of the argument presented in Theorem \ref{approx_vp}, this suggests that the eigenvalues of $A_N$ are close to those of the proposed estimator $\wA$ taken with $\phi_n = \varphi^{(c)}_n$.
	Now, we will focus on the first elements of $T_M$. 
\begin{proposition}
	\label{prop:Hermite_First_elts}
	For all positive real number $c$ and for all positive integer $n \leq c/2$, we have 
	$$ T_M(i,j) = \delta_{i,j}+R_{i,j}(c), $$
	where $|R_{i,j}(c)| \leq \frac{ (2c)^{\frac{l+j-2}{2}}}{\sqrt{\pi c} ((l-1)!)^{1/2}((j-1)!)^{1/2}} e^{-c}  $
\end{proposition}
\begin{proof}
	We recall the following estimate in \cite{BJK}. For any positive integer $k$ and for any positive real number $a$ satisfying $ k \leq \frac{a^2-1}{2}$, we have 
	\begin{equation}\label{bound-phi}
		\int_{|t|>a} \big|\varphi_n(t)\big|^2 dt \leq \frac{2^{n+1}}{\sqrt{\pi}n!} a^{2n-1}e^{-a^2}.
	\end{equation}
	This implies, 
	\begin{eqnarray}
		\norm{\varphi^{(c)}_n-B_c\varphi^{(c)}_n}_{L^2(\R)} &=& \frac{1}{2\pi}\int_{|t|>c} \Big|\mathcal{F}\big[\varphi^{(c)}_n\big](t)\Big|^2 dt = \frac{1}{2\pi \sqrt{c}}\int_{|t|>c} \big|\mathcal{F}[\varphi_n](t/\sqrt{c})\big|^2 dt \nonumber \\
		&=& \frac{1}{2\pi} \int_{|t|>\sqrt{c}} \Big|\varphi_n(t)\Big|^2 dt
		\leq \frac{1}{2\pi}\frac{2^{n+1}}{\sqrt{\pi c}n!} c^{n}e^{-c}.
	\end{eqnarray}
	and, 
	\begin{equation}
		\label{ineq}
		\norm{\varphi^{(c)}_n-D\varphi^{(c)}_n }_{L^2(\R)} = \int_{|t|>1} \big|\varphi^{(c)}_n(t)\big|^2 dt = \frac{1}{2\pi} \int_{|t|>\sqrt{c}} \Big|\varphi_n(t)\Big|^2 dt
		\leq \frac{1}{2\pi}\frac{2^{n+1}}{\sqrt{\pi c}n!} c^{n}e^{-c}.
	\end{equation}
	Let
	\begin{eqnarray}
		R_{l,m}(c) & = & <DB_cD \varphi^{(c)}_{l-1},D\varphi^{(c)}_{m-1}>_{L^2(\R)} - <D \varphi^{(c)}_{l-1},D \varphi^{(c)}_{m-1}>_{L^2(\R)}  \\
		&=& <B_cD \varphi^{(c)}_{l-1},D \varphi^{(c)}_{m-1}>_{L^2(\R)} - <D \varphi^{(c)}_{l-1},D \varphi^{(c)}_{m-1}>_{L^2(\R)} \nonumber \\
		&=& <B_cD  \varphi^{(c)}_{l-1}-D \varphi^{(c)}_{l-1},D \varphi^{(c)}_{m-1}>_{L^2(\R)}  \nonumber
	\end{eqnarray}
	Then, 
	\begin{eqnarray}
		\left|R_{l,m}(c) \right| &\leq & \norm{B_cD \varphi^{(c)}_{l-1}-D \varphi^{(c)}_{l-1}}_{L^2(\R)}\norm{D \varphi^{(c)}_{m-1}}_{L^2(\R)} \nonumber \\
		&\leq& \frac{ (2c)^{l-1}}{\sqrt{\pi c} (l-1)!} e^{-c} = R_l(c).
		\label{reste}
	\end{eqnarray}
	On the other hand,
	\begin{eqnarray}
		<D \varphi^{(c)}_{l-1},D \varphi^{(c)}_{m-1}>_{L^2(\R)} &=& \int_{-1}^1 \varphi_{l-1}^{(c)}(x)\varphi_{m-1}^{(c)}(x) dx  \nonumber \\ &=&\int_{\R} \varphi_{l-1}^{(c)}(x)\varphi_{m-1}^{(c)}(x) dx - \int_{|x|>1} \varphi_{l-1}^{(c)}(x)\varphi_{m-1}^{(c)}(x) dx \nonumber \\
		&=& \delta_{i,j} - \int_{|x|>1} \varphi_{l-1}^{(c)}(x)\varphi_{m-1}^{(c)}(x) dx \nonumber \\
		&=& \delta_{i,j} + (R_{l}(c))^{1/2}(R_{m}(c))^{1/2}
	\end{eqnarray}  
	The last equality follows from the Cauchy-Schwarz inequality combined with \eqref{bound-phi}.
	
\end{proof}

 From Proposition~\ref{prop:Hermite_First_elts} and Figure~\ref{hermite_err}, we deduce that a convenient choice for the truncation parameter in this case is $M=[c/2]$ and that  

 	\[T_{M}= \frac{\pi}{c} I_M
 	\]
where $I_M$ is the square identity matrice of size $M \times M$. Therefore, we obtain the following approximation:
\begin{equation}
	\label{eq:mat_approx}
	\wA \approx \frac{\pi}{c}
\begin{bmatrix}
	H_M H_M^T & 0_{(N-M)\times(N-M)} \\
	0_{(N-M)\times(N-M)} & 0_{(N-M)\times(N-M)}
\end{bmatrix}
\end{equation}
where $H_M$ results from the truncation of the matrix $H$ to its first $M$ lines.
Collecting these results, we can consider that the matrix $A_N$ has $N-M$ eigenvalues almost equal to  zero.

\begin{lemma}\label{properties}
	In the case where the $(X_i)_{1 \leq i \leq N}$ follow the uniform law on $[-1,1]$, the matrix $H$ expressed in the scaled Hermite basis has the following properties. 
	\begin{itemize}
		\item $ \E[H(i,m)] = (R_{m}(c))^{1/2}, \quad \forall m \geq 2.$  $m,n,i,j \in \N$ such that $max\{m,n\} \geq 2$ and $i \neq j$. 
		\item  $\E[H(i,m)H(i,n)] =  \delta_{m,n} + (R_{n}(c))^{1/2}(R_{m}(c))^{1/2} \quad \forall m,n. $
	\end{itemize}
Here, the quantity $R_m(c)$ is defined in \eqref{reste}.
\end{lemma}

We are now ready to state the theorem on the locations of the eigenvalues of the matrix $H_M H_M^T$ and therefore those of the matrix $\wA$.
\begin{theorem}
	\label{th:chernoff}
	Let $G_M = H_MH_M^T$. Then, we have 
	 \begin{equation}\label{eq:lambdamin}
		\PP
		\left(\lambda_{{\min}}\left(G_M\right) \ge (1-\delta)M(1+MR_{M-1}(c)) \right) 
		\ge 1- M \cdot \exp\left(-\frac{\delta^2(1+MR_{M-1}(c))}{2(c\pi)^{1/4}}e^{c/2}\right)
	\end{equation}
	\begin{equation}\label{eq:lambdamax}
		\PP\left(\lambda_{{\max}}\left(G_M\right)
		\le (1+\delta)M(1+MR_0(c)) \right) 
		\ge 1- M \cdot \exp\left(-\frac{\delta^2(1+MR_0(c))}{3(c\pi)^{1/4}}e^{c/2}\right).
	\end{equation}
Here, the quantities $R_0(c)$  and $R_{M-1}(c)$ are defined in \eqref{reste}.	
\end{theorem}

\begin{proof}
	This proof is based on the following Chernoff's theorem: \\
	\noindent 
	{\bf Matrix Chernoff Theorem \cite{Tropp}:} 
	Consider a sequence of $n$ independent $M \times M$
	random Hermitian matrices $\{\pmb{Z}_k\}$. Assume that for some $L>0$, we have 
	$ 0 \preccurlyeq \pmb{Z}_k \preccurlyeq L \cdot \pmb{I}_M\,${{
			\footnote{{{For two symmetric matrices $A$ and $B$ of  the same order, the relation $A \preccurlyeq B$ means that the matrix $B-A$ is positive semi-definite.}}}}}. 
	Let 
	$ \pmb{A}=\sum_{k=1}^n \pmb{Z}_k, \quad \ \mu_{\min}=
	\lambda_{\min}(\mathbb{E}(\pmb{A}))$ {and}
	$ \mu_{\max}=
	\lambda_{\max}(\mathbb{E}(\pmb{A}))\,.
	$
	Then, for any $\delta \in (0,1]$, we have 
	\begin{eqnarray}
		\PP\left(\lambda_{\min}(\pmb{A}) \le (1-\delta) \mu_{\min} \right) &\le& M\cdot  \exp\left(  -\frac{\delta^2 \mu_{\min}}{2L}\right), \nonumber \\ 
		\PP\left(\lambda_{\max}(\pmb{A}) \ge (1+\delta) \mu_{\max} \right) &\le& M \cdot \exp\left(  -\frac{\delta^2 \mu_{\max}}{3L}\right).
	\end{eqnarray}
	Let $G_M = \sum_{k=0}^{M-1} G^{(k)}_M $ with $ G^{(k)}_M (i,j) = \varphi^{(c)}_k(X_i) \varphi^{(c)}_k(X_j)$.
	By Gershgorin's theorem, one has
	 $$ \lambda_{\max} (G^{(k)}_M) \le \max_{1 \le i \le M} 
	\left( \big |\varphi^{(c)}_k(X_i)\big |^2 + \sum_{j \neq i} \big |\varphi^{(c)}_k(X_i)\varphi^{(c)}_k(X_j)\big |  \right) \leq M \sup_{x\in [-1,1]} \big|\varphi^{(c)}_k(x)\big|^2 .
	$$
	Using \cite{NIST} page 450, one concludes that
	\begin{equation}
		\lambda_{\max} (G^{(k)}_M) \le M (c\pi)^{1/4} e^{-c/2}=L(c,M).
	\end{equation}
	Then, $0 \preceq G^{(k)}_M \preceq LI_M$.
	We have now to compute $\E(G_M)$. By Lemma \ref{properties}, one has 
	\begin{equation}
		\E[G_M](i,j) = 
			\begin{cases}
			M + \displaystyle\sum_{k=0}^{M-1}R_{k}(c), \ \mbox{ if }\ i=j ;\\
			\displaystyle \sum_{k=0}^{M-1}R_{k}(c), \ \mbox{ otherwise. }
			\end{cases}
	\end{equation}
	Consequently, we obtain \eqref{eq:lambdamin} and \eqref{eq:lambdamax} using the Matrix Chernoff Theorem.

\end{proof}

\section{Conclusion}
In this paper, we have investigated the distribution of the eigenvalues of the some class of random matrices $A_N$ and especially the sinc kernel case. First, we have shown that the random matrix can be approximated by the product $\wA=H T_M H^T$. We have shown that the eigenvalues of $A_N$ are well approximated by those of $\wA$.
In the case where $c$ is fixed and $N$ is large, we have shown that the eigenvalues distribution is almost $\delta_0$. The remaining nonzero eigenvalues are approximated by those of the related integral operator. In the case of large values of $c$, Theorem \ref{th:chernoff} with \eqref{eq:mat_approx} imply that the eigenvalues are concentrated around zero and one. The frequencey of eigenvalues close to one is proportional to $c$. In order to illustrate these results, Figure \ref{fig:conc} gives an overview for different situations that can happen depending on the values of $c$ respectively to $N$. In the last two cases, we have any theoretical explication why the eigenvalues are concentrated on only two values around one. 

	\begin{figure}[H]
		\centering
		\includegraphics[scale=0.4]{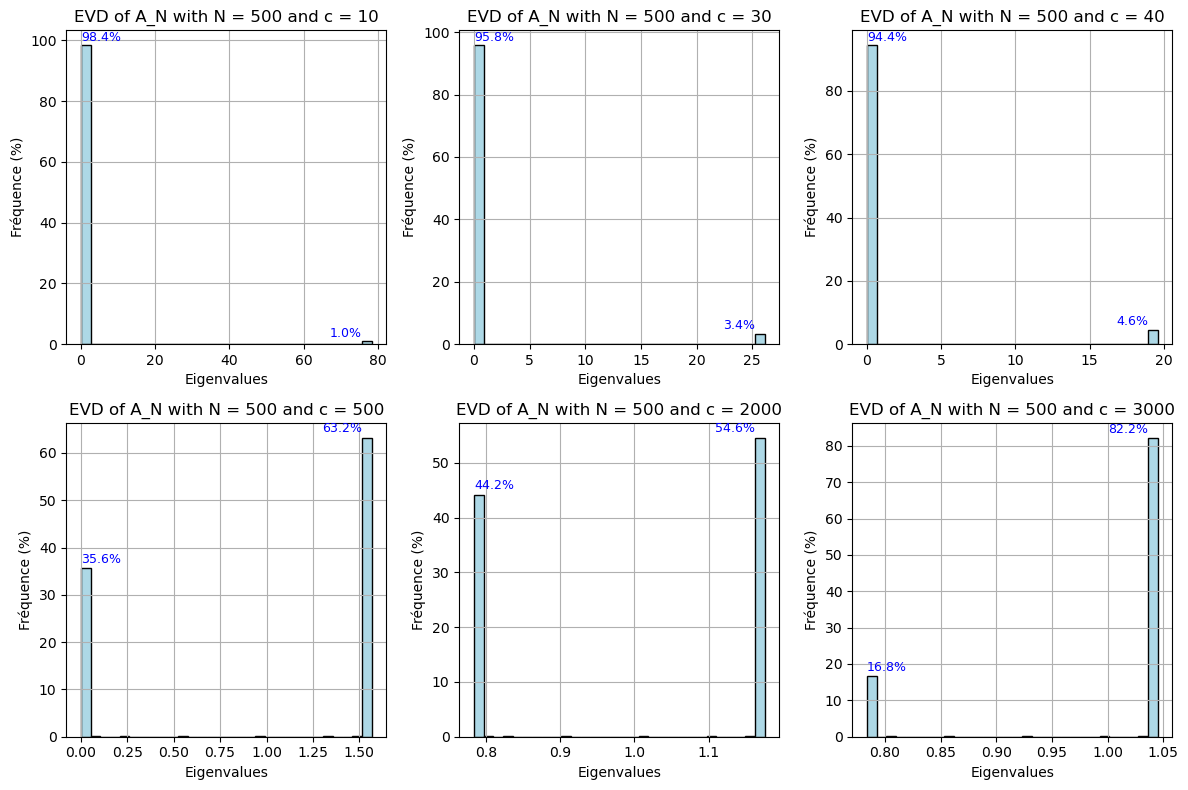}
		\caption{The eigenvalue distribution of the sinc random matrix for different values of $c$.}
		\label{fig:conc}
	\end{figure}

\qed

\noindent 
{\large {\bf Statements and  Declarations}}\\

\noindent
{\bf Conflict of interest}   The authors declare that they have no conflict of interest.\\

\noindent
{\large {\bf Funding} This work was supported by the  
	DGRST  research grant  LR21ES10.\\
	
\noindent {\bf Data availability :} This article has no associated data.

\end{document}